\documentclass[12pt]{article}
\usepackage{amsmath,amssymb, amsthm}
\newtheorem{theorem}{Theorem}[section]

\newtheorem{lemma}[theorem]{Lemma}

\newtheorem{conjecture}[theorem]{Conjecture}
\numberwithin{equation}{section}

\def\E_1{{{\mathcal{P}_1}}}
\def\E{{{\mathcal{E}}}}

\begin{document}
\title{On the determinants and permanents of matrices with restricted entries over prime fields  }

\author{
Doowon Koh\thanks{Department of Mathematics, Chungbuk National University. Email: {\tt koh131@chungbuk.ac.kr}}
\and 
    Thang Pham\thanks{Department of Mathematics,  UCSD. Email: {\tt v9pham@ucsd.edu}}
  \and
  Chun-Yen Shen \thanks{Department of Mathematics,  National Taiwan University. Email: {\tt cyshen@math.ntu.edu.tw}}
\and 
    Le Anh Vinh \thanks{Department of Mathematics, Vietnam National University.
    Email: {\tt vinhla@vnu.edu.vn
}}}

\date{}
\maketitle

\begin{abstract}
Let $A$ be a set in a prime field $\mathbb{F}_p$. In this paper, we prove that $d\times d$ matrices with entries in $A$ determine almost $|A|^{3+\frac{1}{45}}$ distinct determinants and almost $|A|^{2-\frac{1}{6}}$ distinct permanents when $|A|$ is small enough. 
\end{abstract}
\section{Introduction}

Throughout the paper, let $q=p^r$ where $p$ is an odd prime and $r$ is a positive integer. Let $\mathbb{F}_q$ be a finite field with $q$ elements. The prime base field $\mathbb{F}_p$ of $\mathbb{F}_q$ may then be naturally identified with $\mathbb{Z}_p = \mathbb{Z}/p\mathbb{Z}$.

Let $M=[a_{ij}]$ be an $n \times n$ matrix. Two basic parameters of $M$ are its determinant 
\[ \mbox{Det}(M) := \sum_{\sigma \in S_n} \mbox{sgn}(\sigma) \prod_{i = 1}^n a_{i\sigma(i)} \] 
and its permanent 
\[ \mbox{Per}(M) := \sum_{\sigma \in S_n} \prod_{i = 1}^n a_{i\sigma(i)},\]
where $S_n$ is the symmetric group on $n$ elements. 

For a positive integer $d$, let ${M}_{d} (A)$ denote
the set of $d \times d$ matrices with components in the set $A$. For a given $t$ in the field, let $D_d(A,t)$ and $P_d(A,t)$ be the number of matrices in ${M}_{d}(A)$ having determinant $t$ and permanent 
$t,$ respectively. Let $f_d(A)$ and $g_d(A)$ be the number of distinct determinants and distinct permanents determined by matrices in ${M}_d(A),$ respectively.

In \cite{igor}, Ahmadi and Shparlinski studied some classes of matrices over the prime field $\mathbb{F}_p$ of $p$ elements with components in a
given interval 
$$[- H, H] \subset [- (p - 1) / 2, (p - 1) / 2].$$ 
They proved some distribution results on the number of $d \times d$ matrices with entries in a given inteval having a fixed determinant.  More precisely, they obtained that 
\[ D_d([-H,H],t) = (1+o(1)) \frac{(2H+1)^{d^2}}{p}\] 
if $t \in \mathbb{F}_p^*$ and $H \gg p^{3/4}$, which is asymptotically close to the expected value.  In the case $d = 2$, the lower bound can be improved to $H \gg p^{1/2 + \epsilon}$ for any constant $\epsilon > 0$. Recall that the notation $U=O(V)$ and $U \ll V$ are equivalent to the assertion that the inequality $|U| \leq cV$ holds for some constant $c>0$. Note that the implied constants in the symbols $O, o$ and $\ll$ may depend on integer parameter $d$. We also will use the notation $U \gtrsim V$ for the case $U\gg (\log U)^{-c} V$ for some positive constant $c$.

Covert et al. \cite{covert} studied this problem in a more general setting, namely, they proved that for any $t\in \mathbb{F}_q^*$ and $A\subset \mathbb{F}_q$, the number of matrices in $M_3(A)$ of determinant $t$ satisfies
\[ D_3(A,t) = (1+o(1)) \frac{|A|^9}{q}.\] 
In \cite{vinh-det}, the fourth listed author extended this result to higher dimensional cases. More precisely, he proved the following: 
\[ D_d(A,t) = (1+o(1)) \frac{|A|^{d^2}}{q}\] 
for any $t \in \mathbb{F}_q^*$ and $A \subset \mathbb{F}_q$ of cardinality $|A| \gg q^{\frac{d}{2d-1}}$.

Another important question is to ask for the number of distinct determinants $f_d(A)$ determined by matrices in ${M}_d(A)$. The authors of \cite{covert} showed that $f_4(A) = q$ whenever $|A| > \sqrt{q}$. Their result can also be extended to higher dimensions.

For the permanant, the fourth listed author \cite{vinht} obtained several results for the distribution of a given permanent and the number of distinct permanents determined by matrices in ${M}_d(A)$. More precisely, he showed that $g_d(A) = (1 + o(1))q$ if $A \subset \mathbb F_q$ with cardinality $|A| \gg q^{\frac{1}{2} + \frac{1}{2d-1}}$. Furthermore, if we restrict our study to matrices over the prime field $\mathbb{F}_p$ with components in a given interval $I := [a+1, a+b] \subset \mathbb{F}_p$, we obtain a stronger result 
\[ D_d(I,t) = (1+o(1)) \frac{b^{d^2}}{p}\]
if $b \gg p^{1/2 + \epsilon}$ for any constant $\epsilon > 0$. We refer the reader to \cite{vinht} for more details.

The main purpose of this paper is to study the the number of distinct determinants and permanents determined by matrices in ${M}_d(A)$ when $A$ is a small subset of $\mathbb{F}_p$. More precisely, we have the following results for the number of distinct determinants.

\begin{theorem}\label{thm-det1}
Let $A$ be a set in $\mathbb{F}_p$. 
\begin{itemize}
\item[(i)] If $ |A|\le p^{2/3},$ then we have 
\[f_2(A)\gg |A|^{3/2}.\]
\item[(ii)]  If $|A|\le p^{\frac{45 \times 2^{d/2}}{136 \times2^{d/2}-137}}$ and $d\ge 4$ even, we have 
\[f_d(A)\gtrsim |A|^{3+\frac{1}{45}-\frac{137}{45 \times 2^{d/2}}}.\]
\end{itemize}
\end{theorem}

\begin{theorem}\label{thm-det2}
Let  $A$ be a set in $\mathbb{F}_p$ and $d\ge 3$ odd. 
\begin{itemize}
\item[(i)] If $ |A|\le p^{4/7}$ and $d=3$, we have 
\[f_3(A)\gg |A|^{7/4}.\]
\item[(ii)]  If $|A|\le p^{\frac{45 \times 2^{(d-1)/2}}{136 \times2^{(d-1)/2}-137}}$ and $d\ge 5$ odd, we have 
\[f_d(A)\gtrsim |A|^{\frac{5}{2}+\frac{1}{90}-\frac{137}{45\times 2^{(d+1)/2}}}.\]
\end{itemize}
\end{theorem}
From the lower bounds of Theorems \ref{thm-det1} and \ref{thm-det2}, we make the following conjecture. 
\begin{conjecture}
For $A\subset \mathbb{F}_p$, suppose $d$ is large enough, we have
\[f_d(A)\gtrsim \min\left\lbrace |A|^{4}, p\right\rbrace.\]
\end{conjecture}
For the number of distinct permanents, we have the following results.

\begin{theorem}\label{thm-2-per}
Let $A$ be a set in $\mathbb{F}_p$ with $|A|\le p^{2/3}$. We have
\[g_2(A)\gg |A|^{3/2}.\]
\end{theorem}
\begin{theorem}\label{thmper}
Let $A$ be a set in $\mathbb{F}_p$ with $|A|\le p^{1/2}$. For any integer $d\ge 3$, we have
\[g_d(A) \gg  |A|^{2-\frac{1}{6}-\frac{1}{3}\left(\frac{2}{5}\right)^{d-2}} .\]
\end{theorem}
If $A$ is a set in an arbitrary finite field $\mathbb{F}_q$ where $q$ is an odd prime power, then it has been shown by Vinh \cite{vinht} that under the condition $|A|\ge q^{\frac{d}{2d-1}}$  the matrices in $M_d(A)$ determine a positive proportion of all permanents. The same threshold, i.e. $\frac{d}{2d-1}$, is indicated to be true for the Erd\H{o}s distinct distances problem in $A^d$ over $\mathbb{F}_q^d$ (see \cite{hieu}). Recently, Pham, Vinh, and De Zeeuw \cite{vinhtt} showed that for $A\subset \mathbb{F}_p$, the number of distinct distances determined by points in $A^d$ is almost $|A|^2$ if the size of $A$ is not so large. Thus it seems reasonable to make the following conjecture. 
\begin{conjecture}
For $A\subset \mathbb{F}_p$ and an integer $d\ge 2$, we have 
\[g_d(A)\ge \min\left\lbrace |A|^2, p\right\rbrace.\]
\end{conjecture}

Recently, another question on determinants of matrices has been studied by Karabulut \cite{ye} by employing spectral graph theory techniques. More precisely, she showed that for a set $\mathcal{E}$ of $2\times 2$ matrices over $\mathbb{F}_p$, if $|\mathcal{E}|\gg p^{5/2}$, then for any $\lambda\in \mathbb{F}_p^*$ there exist two matrices $X, Y\in \mathcal{E}$ such that $\mbox{Det}(X-Y)=\lambda$. In this paper, we give a result on the case $\mathcal{E}=M_2(A)$ for some small set $A\subset \mathbb{F}_p$.  For $A\subset\mathbb{F}_p$, we define
\[F_2(A):=\left\lbrace  \mbox{Det}(X-Y)\colon X, Y\in M_2(A)\right\rbrace, ~~G_2(A):=\left\lbrace  \mbox{Per}(X-Y)\colon X, Y\in M_2(A)\right\rbrace.\]
\begin{theorem}\label{m-dif}
For $A\subset \mathbb{F}_p$ with $|A|\le p^{9/16}$, we have 
\[|F_2(A)|, ~|G_2(A)| \gtrsim |A|^{\frac{7}{4}+\frac{1}{60}}.\]
\end{theorem}

\section{Proofs of Theorems \ref{thm-det1} and \ref{thm-det2}}
To prove our main theorems, we shall make use of the following lemmas.
\begin{lemma}[\cite{rudnev}, Corollary 4]\label{t:lm}
For $A\subset \mathbb{F}_p$, we have 
\[|AA\pm AA|\gg \min \left\lbrace |A|^{3/2}, ~p\right\rbrace.\]
\end{lemma}
\begin{lemma}[\cite{frank}, Corollary 11]\label{t:lm2}
Suppose that none of  $B, C, D\subset \mathbb{F}_p$ is the same as $\{0\}.$ Then we have 
\[~~|B(C-D)|\gg \min \left\lbrace |B|^{1/2}|C|^{1/2}|D|^{1/2}, ~ p\right\rbrace\]
\end{lemma}

\begin{lemma}[\cite{murphy}, Theorem 27]\label{t:lm4} 
For $A\subset \mathbb{F}_p$, we have 
\[|(A-A)(A-A)|\gtrsim \min\{|A|^{3/2+1/90},~ p\}\]
\end{lemma}

\begin{lemma}\label{lm11}
For $A\subset \mathbb{F}_p$ with $|A|\le p^{45/68}$ and $d\ge 4$ even, we have 
\[f_d(A)\gtrsim \min\left\lbrace f_{d-2}(A)^{1/2}|A|^{3/2+1/90}, ~p\right\rbrace.\]
\end{lemma}
\begin{proof} We may assume that $|A|\ge 2.$
Let  $X_d$ be the set of distinct determinants of matrices in $M_d(A)$. Let $M$ be a $d\times d$ matrix in $M_d(A)$  with the following form 
\[
M=\begin{bmatrix}
    a_{11} & a_{12} & a_{33} & \dots  & a_{3d-1}& a_{3d} \\
    a_{21} & a_{22} & a_{43} & \dots  & a_{4d-1} & a_{4d}\\
    u_1&u_2&a_{33} & \dots  & a_{3d-1}& a_{3d}\\
    v_1&v_2&a_{43} & \dots  & a_{4d-1} & a_{4d}\\
    a_{51} & a_{52} & a_{53} & \dots  & a_{5d-1}& a_{5d} \\
    \vdots & \vdots & \vdots & \ddots & \ddots & \vdots \\
    a_{d-11} & a_{d-12} & a_{d-13} & \dots  & a_{d-1d-1}& a_{d-1d}\\
 a_{d1} & a_{d2} & a_{d3} & \dots  & a_{dd-1}& a_{dd}\\
\end{bmatrix}.
\]
We have 
\begin{align*}
\mbox{Det}(M)&=\left| \begin{matrix}
    a_{11}-u_1 & a_{12}-u_2 & 0 & \dots  & 0& 0 \\
    a_{21}-v_1 & a_{22}-v_2 & 0 & \dots  & 0 & 0\\
    u_1&u_2&a_{33} & \dots  & a_{3d-1}& a_{3d}\\
    v_1&v_2&a_{43} & \dots  & a_{4d-1} & a_{4d}\\
    a_{51} & a_{52} & a_{53} & \dots  & a_{5d-1}& a_{5d} \\
    \vdots & \vdots & \vdots & \ddots & \ddots & \vdots \\
    a_{d-11} & a_{d-12} & a_{d-13} & \dots  & a_{d-1d-1}& a_{d-1d}\\
 a_{d1} & a_{d2} & a_{d3} & \dots  & a_{dd-1}& a_{dd}\\
\end{matrix}\right|\\ 
&=\left| \begin{matrix}
a_{11}-u_1&a_{12}-u_2\\
a_{21}-v_1&a_{22}-v_2
\end{matrix}\right|\cdot \left| \begin{matrix}
    a_{33} & \dots  & a_{3d-1}& a_{3d}\\
    a_{43} & \dots  & a_{4d-1} & a_{4d}\\
    a_{53} & \dots  & a_{5d-1}& a_{5d} \\
    \vdots & \vdots  & \vdots & \vdots \\
     a_{d-13} & \dots  & a_{d-1d-1}& a_{d-1d}\\
  a_{d3} & \dots  & a_{dd-1}& a_{dd}\\
\end{matrix}\right|.
\end{align*}
This implies that 
\[X_{d-2}\cdot \left((A-A)(A-A)-(A-A)(A-A)\right)\subset X_d.\]
Using Lemmas \ref{t:lm2} and \ref{t:lm4}, we see that
\begin{equation}\label{Xdformula} f_d(A) \gtrsim \min\left\{ f_{d-2}(A)^{1/2} \min\{ |A|^{\frac{3}{2}+\frac{1}{90}},~p\},~p\right\}.\end{equation}
Thus the lemma follows from the assumption that $|A|\le p^{45/68}$ which implies that $\min\{ |A|^{\frac{3}{2}+\frac{1}{90}},~p\} =|A|^{\frac{3}{2}+\frac{1}{90}}.$
\end{proof}
\begin{proof}[Proof of Theorem \ref{thm-det1}] 
 Let $M$ be a $2\times 2$ matrix in $M_2(A)$ of the following form 
\[M=
\begin{bmatrix}
    a&b\\
    c&d\\
\end{bmatrix}.
\]
Then we have $\mathtt{det(M)}=ad-bc$. This implies that $f_2(A)=|AA-AA|$. Thus the first part of Theorem \ref{thm-det1} follows from Lemma \ref{t:lm}. 

In order to prove the second part of Theorem \ref{thm-det1}, we use induction on $d\ge 4$ even. In the base case when $d=4,$ the statement follows by combining \eqref{Xdformula} with the first part of Theorem \ref{thm-det1}.
Suppose the statement holds for $d-2\ge 4.$ We now show that it also holds for $d$. Indeed, from Lemma \ref{lm11}, we see that if  $|A|\le p^{45/68},$ then 
\[f_d(A)\gtrsim \min\left\lbrace f_{d-2}(A)^{1/2}|A|^{3/2+1/90}, ~p\right\rbrace.\]
By induction hypothesis, it follows that if  $|A|\le p^{\frac{45 \times 2^{(d-2)/2}}{136 \times2^{(d-2)/2}-137}},$  then
\[f_{d-2}(A)\gtrsim |A|^{3+\frac{1}{45}-\frac{137}{45 \times 2^{(d-2)/2}}}.\]
By the above two inequalities, we see  that  if $|A|\le p^{\frac{45 \times 2^{(d-2)/2}}{136 \times2^{(d-2)/2}-137}},$ then
\[f_d(A)\gtrsim \min\left\lbrace |A|^{3+\frac{1}{45}-\frac{137}{45 \times 2^{d/2}}}, ~p\right\rbrace.\]
By a direct comparison, this clearly implies that if  $|A|\le p^{\frac{45 \times 2^{d/2}}{136 \times2^{d/2}-137}},$ then 
\[f_d(A)\gtrsim \min\left\lbrace |A|^{3+\frac{1}{45}-\frac{137}{45 \times 2^{d/2}}}, ~p\right\rbrace = |A|^{3+\frac{1}{45}-\frac{137}{45 \times 2^{d/2}}}.\]
Hence the proof of the theorem is complete.
\end{proof}
In order to prove Theorem \ref{thm-det2}, we need the following result.
\begin{lemma}\label{t:lm3}
Let $A$ be a set in $\mathbb{F}_p$  and $d\ge 3$ odd. We have
\[f_d(A)\gg \min \left\lbrace f_{d-1}(A)^{1/2}|A|,~ p\right\rbrace.\]
\end{lemma}
\begin{proof} We may assume that $|A|\ge 2,$ because the statement of the lemma is obvious for $|A|=1.$ Hence, we are able to invoke Lemma \ref{t:lm2}.
Let  $X_d$ be the set of distinct determinants of matrices in $M_d(A)$. Let $M$ be a $d\times d$ matrix in $M_d(A)$ of the following form 
\[M=\begin{bmatrix}
    a_{11} & a_{12} & a_{13} & \dots  & a_{1d-1}& a_{11} \\
    a_{21} & a_{22} & a_{23} & \dots  & a_{2d-1} & a_{21}\\
    \vdots & \vdots & \vdots & \ddots & \ddots & \vdots \\
    a_{d-11} & a_{d-12} & a_{d-13} & \dots  & a_{d-1d-1}& a_{d-11}\\
    x_{1} & x_{2} & x_{3} & \dots  & x_{d-1}& x_{d}
\end{bmatrix}.\]
We expands the last low. Then the basic properties of determinants yield 
\begin{align*}\mbox{Det}(M)=&(-1)^{d+1}x_1 \left|\begin{matrix}
     a_{12} & a_{13} & \dots  & a_{1d-1}& a_{11} \\
     a_{22} & a_{23} & \dots  & a_{2d-1} & a_{21}\\
     \vdots & \vdots & \ddots & \ddots & \vdots \\
     a_{d-12} & a_{d-13} & \dots  & a_{d-1d-1}& a_{d-11}
\end{matrix}\right|\\
&+x_d \left|\begin{matrix}
    a_{11} & a_{12} & a_{13} & \dots  & a_{1d-1}\\
    a_{21} & a_{22} & a_{23} & \dots  & a_{2d-1}\\
    \vdots & \vdots & \vdots & \ddots & \ddots   \\
    a_{d-11} & a_{d-12} & a_{d-13} & \dots  & a_{d-1d-1}
\end{matrix}\right|\\
=&(x_d-x_1)\left|\begin{matrix}
    a_{11} & a_{12} & a_{13} & \dots  & a_{1d-1}\\
    a_{21} & a_{22} & a_{23} & \dots  & a_{2d-1}\\
    \vdots & \vdots & \vdots & \ddots & \ddots   \\
    a_{d-11} & a_{d-12} & a_{d-13} & \dots  & a_{d-1d-1}
\end{matrix}\right|.\end{align*}
This implies that $(A-A)X_{d-1}\subset X_d.$ Hence, the lemma follows immediately from Lemma \ref{t:lm2}.
\end{proof}
\begin{proof}[Proof of Theorem \ref{thm-det2}]
Let $d\ge 3$ odd. Then $d-1$ is even. Thus
combining Theorem \ref{thm-det1} and Lemma \ref{t:lm3}, we see that
\begin{itemize}
\item[(i)] if $|A|\le p^{2/3}$ we have 
\[f_3(A)\gg \min\{f_2(A)^{1/2} |A|,~p\} \gg \min\{|A|^{7/4},~p\},\]
\item[(ii)]  if $|A|\le p^{\frac{45 \times 2^{(d-1)/2}}{136 \times2^{(d-1)/2}-137}}$ and $d\ge 5$ odd, then we have 
\[f_d(A)\gg \min\{f_{d-1}(A)^{1/2} |A|,~p\}\gtrsim \min\{ |A|^{\frac{5}{2}+\frac{1}{90}-\frac{137}{45\times 2^{(d+1)/2}}},~p\}\]
\end{itemize}
Since $ p^{4/7} < p^{2/3},$ the statement $(\mbox{i})$ implies that
if $|A|\le p^{4/7},$ then 
$f_3(A)\gg \min\{|A|^{7/4},~p\}=|A|^{7/4},$
which completes the proof of the first part of Theorem \ref{thm-det2}.

To prove the second part of Theorem \ref{thm-det2}, first observe that 
$$ \min\left\{ |A|^{\frac{5}{2}+\frac{1}{90}-\frac{137}{45\times 2^{(d+1)/2}}},~p\right\}=|A|^{\frac{5}{2}+\frac{1}{90}-\frac{137}{45\times 2^{(d+1)/2}}}\quad \mbox{for}\quad |A|\le p^{\frac{45\times 2^{(d+1)/2}} {113 \times 2^{(d+1)/2}-137}}, $$
and
$$ p^{\frac{45 \times 2^{(d-1)/2}}{136 \times2^{(d-1)/2}-137}} \le p^{\frac{45\times 2^{(d+1)/2}} {113 \times 2^{(d+1)/2}-137}} \quad \mbox{for odd}~~d\ge 5.$$
The statement of the second part of Theorem \ref{thm-det2}  follows by these observations and the statement \mbox{(ii)} above.
\end{proof}
\section{Proofs of Theorems \ref{thm-2-per} and  \ref{thmper}}
\begin{proof}[Proof of Theorem \ref{thm-2-per}]
Let $M$ be a $2\times 2$ matrix in $M_2(A)$ of the following form 
\[M=
\begin{bmatrix}
    a&b\\
    c&d\\
\end{bmatrix}.
\]
Then we have $\mbox{Per}(M)=ad+bc$. This implies that 
\[g_2(A)\gg |AA+AA|\gg \min\{ |A|^{3/2}, ~p\},\]
where the second inequality follows from Lemma \ref{t:lm}.
Thus if $|A|\le q^{2/3}$, then $g_2(A)\gg |A|^{3/2}.$
This completes the proof of Theorem \ref{thm-2-per}.
\end{proof}
In order to prove Theorem \ref{thmper}, we need the following lemmas.
\begin{lemma}[\cite{frank}, Theorem 4]\label{lienthuoc}
Let $P_1, P_2\subset \mathbb{F}_p$ with $|P_1|\le |P_2|,$ and let $\mathcal{L}$ denote a finite set of lines  in $\mathbb{F}_p^2.$ Assume that 
$|P_1||P_2|^2\le |\mathcal{L}|^3$ and $|P_1| |\mathcal{L}|\ll p^2.$ Then the number of incidences between $P_1\times P_2$ and lines in $\mathcal{L}$, denoted by $I(P_1\times P_2, \mathcal{L})$, satisfies 
\[I(P_1\times P_2, \mathcal{L})\ll |P_1|^{3/4}|P_2|^{1/2}|\mathcal{L}|^{3/4}+|\mathcal{L}|.\]
\end{lemma}

\begin{lemma}\label{bridge}
For $A, B, C \subset \mathbb{F}_p$ with $|B|, |C|\ge |A|$ and $|A|\le p^{1/2}$, we have
\[|A+B||AC|\gg |A|^{8/5}|B|^{2/5}|C|^{2/5}.\]
\end{lemma}
\begin{proof}
To prove this lemma, we follow the arguments of Stevens and de Zeeuw in \cite[Corollary 9]{frank}. Suppose that 
\begin{equation}\label{assM}|A+B|\le |AC|.\end{equation} 
Since the case $|A+B|\ge |AC|$ can be handled in a similar way, we only provide the proof in the case when \eqref{assM} holds.\\

Set $\mathcal{P}:=(A+B)\times (AC)$. Let $\mathcal{L}$ be the set of lines defined by the equations $y=c(x-b)$ with $c\in C$ and $b\in B$. Without loss of generality, we may assume that $0\notin C.$ Then we have $|\mathcal{P}|=|A+B||AC|$ and $|\mathcal{L}|=|B||C|$.
It is clear that the number of incidences between $\mathcal{P}$ and $\mathcal{L}$ is at least $|A||B||C|,$ because each line $y=c(x-b)$ for $(c,b)\in C\times B$ contains the points of the form $(a+b, ac)\in \mathcal{P}$ for all $a\in A.$ In order words, we have
\begin{equation}\label{incidencelow}
|A||B||C| \le I(\mathcal{P}, \mathcal{L}).
\end{equation}

 In order to find an upper bound of $I(\mathcal{P}, \mathcal{L}),$ we now apply Lemma \ref{lienthuoc}  with $P_1=A+B, ~ P_2=AC,$ and $|\mathcal{L}|=|B||C|,$ but we first need to check  its conditions 
\begin{equation} \label{con2} |A+B||AC|^2\le |B|^3|C|^3 \quad \mbox{and}\quad |A+B||B||C| \ll p^2.\end{equation}

Assumet that  $|A+B||AC|^2> |B|^3|C|^3,$ which is the case when the first condition in \eqref{con2} does not hold. Then we have  $|A+B|^2|AC|^2> |B|^3|C|^3$, which implies that 
\[|A+B||AC|>|B|^{3/2}|C|^{3/2} \ge|A|^{11/5}|B|^{2/5}|C|^{2/5} > |A|^{8/5}|B|^{2/5}|C|^{2/5},\]
where the second inequality above follows from the assumption of Lemma  \ref{bridge} that $|B|, |C|\ge |A|.$ 
Thus, to complete the proof of  Lemma \ref{bridge}, we may assume that $|A+B||AC|^2\le |B|^3|C|^3,$ which is the first condition in \eqref{con2}.

Next, we shall show that we may assume the second condition in \eqref{con2} to prove Lemma \ref{bridge}. 
Since $|A+B||AC|\ge |B||C|,$ we see that  if  $|B||C|> |A|^{8/5}|B|^{2/5}|C|^{2/5}$, then the conclusion of Lemma \ref{bridge} holds. We also see that  the conclusion of Lemma \ref{bridge} holds if $|A+B|> |A|^{4/5}|B|^{1/5}|C|^{1/5}$, as we have assumed that $|AC|\ge |A+B|$ in \eqref{assM}. Hence, to prove Lemma \ref{bridge}, we may assume that $|B||C|\le |A|^{8/5}|B|^{2/5}|C|^{2/5}$ (namely, $|B||C|\le |A|^{8/3}$) and $|A+B|\le |A|^{4/5}|B|^{1/5}|C|^{1/5}$. These conditions imply that 
\[|A+B||B||C|\ll |A|^{4/5}|B|^{6/5}|C|^{6/5}\ll |A|^{20/5}\ll p^2,\]
where the last inequality follows from the assumption of Lemma \ref{bridge} that $|A|\le p^{1/2}$. Therefore, to prove Lemma \ref{bridge}, we may assume the second condition in  \eqref{con2}.  

In conclusion, by \eqref{assM} and \eqref{con2}, we are able to apply Lemma \ref{lienthuoc} so that  we obtain   that
\[|A||B||C|\le I(\mathcal{P}, \mathcal{L})\ll |A+B|^{3/4}|AC|^{1/2}|B|^{3/4}|C|^{3/4} + |B||C|,\]
where we recall that the first inequality is given in \eqref{incidencelow}. 
This leads to the following
\[|A||B|^{1/4}|C|^{1/4} \ll |A+B|^{3/4} |AC|^{1/2}.\]
By \eqref{assM}, the above inequality implies that
\[|A+B||AC|\gg |A|^{8/5}|B|^{2/5}|C|^{2/5},\]
which completes the proof of Lemma \ref{bridge}.
\end{proof}
\begin{lemma}\label{per-chinh}
Let $A$ be a set in $\mathbb{F}_p$ with $|A|\le p^{1/2}$. Then, for any integer $d\ge 2$, we have
\[|A^d||dA|\gg |A|^{\frac{8}{3}-\frac{2}{3}\left(\frac{2}{5}\right)^{d-1}},\]
where $A^d=A\cdots A ~(d ~\mbox{times})$, and $dA=A+\cdots+A~ (d\,\mbox{ times}).$
\end{lemma}
\begin{proof}
We prove this lemma by induction on $d$. The base case $d=2$ follows immediately from Lemma \ref{bridge} with $B=C=A$. Suppose that the statement holds for $d-1\ge 2.$ We now show that it also holds for $d$. Indeed, from Lemma \ref{bridge} we see that
\[|A^d||dA|\gg |A|^{8/5}(|A^{d-1}||(d-1)A|)^{2/5}.\]
By induction hypothesis, we obtain
\[|A^{d-1}||(d-1)A|\gg |A|^{\frac{8}{3}-\frac{2}{3}\left(\frac{2}{5}\right)^{d-2}},\]
This implies that
\[|A^d||dA|\gg |A|^{8/5}(|A^{d-1}||(d-1)A|)^{2/5}\gg |A|^{\frac{8}{3}-\frac{2}{3}\left(\frac{2}{5}\right)^{d-1}},\]
which concludes the proof of the lemma.
\end{proof}
We are now ready to give a proof of Theorem \ref{thmper}.
\begin{proof}[Proof of Theorem \ref{thmper}]
Let $M$ be a $d\times d$ matrix in $M_d(A)$ of the following form 
\[
\begin{bmatrix}
    x_1 & x_1 & x_1 & \dots  & x_1& x_1 \\
    x_2 & x_2 & x_2 & \dots  & x_2 & x_2\\
    x_3 & x_3 & x_3 & \dots  & x_3& x_3 \\
    \vdots & \vdots & \vdots & \vdots & \vdots & \vdots \\
    x_{d-1} & x_{d-1} & x_{d-1} & \dots  & x_{d-1}& x_{d-1}\\
 x_{d1} & x_{d2} & x_{d3} & \dots  & x_{dd-1}& x_{dd}\\
\end{bmatrix}.
\]
We have $\mbox{Per}(M)=(d-1)! (x_1\cdots x_{d-1})(x_{d1}+\cdots+x_{dd})$. This implies that
\[g_d(A)\ge |A^{d-1}\cdot\left((d-1)A+A\right)|.\]
From Lemma \ref{t:lm2}, we have
\[g_d(A)\ge |A^{d-1}\cdot((d-1)A+A)| \gg \min\left\{|A|^{1/2}\left(|A^{d-1}||(d-1)A|\right)^{1/2},~p\right\}.\]
 
From Lemma \ref{per-chinh} for $d-1,$ the above inequality implies that
\begin{align*}g_d(A) &\gg \min \left\{ |A|^{\frac{1}{2}}\left(|A|^{\frac{8}{3}-\frac{2}{3}\left(\frac{2}{5}\right)^{d-2}}\right)^{1/2},~p\right\}\\
&=\min\left\{|A|^{2-\frac{1}{6}-\frac{1}{3}\left(\frac{2}{5}\right)^{d-2} } ,~p\right\}=|A|^{2-\frac{1}{6}-\frac{1}{3}\left(\frac{2}{5}\right)^{d-2} },
\end{align*}
where the last equality follows by the assumption of Theorem \ref{thmper} that $|A|\le p^{1/2}.$
Thus the proof of Theorem \ref{thmper} is complete. \end{proof}
\section{Proof of Theorem \ref{m-dif}}
To prove Theorem \ref{m-dif}, we make use of the following lemmas.
\begin{lemma}[\cite{vinhtt}, Corollary 3.1]\label{m-dif1}
For $X, B\subset \mathbb{F}_p$ with $|X|\ge |B|$. We have 
\[|X\pm B\cdot B|\gg \min\left\lbrace |X|^{1/2}|B|, p\right\rbrace.\]
\end{lemma}

\begin{lemma}[\cite{murphy}, Theorem 2]\label{m-dif2}
For $A\subset \mathbb{F}_p$ with $|A|\le p^{9/16}$, we have 
\[|A-A|^{18}|AA|^9 \gtrsim |A|^{32}.\]
\end{lemma}

We are now ready to prove Theorem \ref{m-dif}.
\begin{proof}[Proof of Theorem \ref{m-dif}]
It is clear that
\[F_2(A)=(A-A)(A-A)-(A-A)(A-A).\]
Suppose $|A-A|\ge |A|^{1+\epsilon}$ where $\epsilon=1/90.$ It follows from Lemmas \ref{m-dif1} and \ref{t:lm4}  with $X=(A-A)(A-A)$ and $B=(A-A)$ that for $|A|\le p^{9/16}$, 
\[|F_2(A)|\gtrsim |A|^{\frac{7}{4}+\frac{1}{180}+\epsilon},\]
and we are done. Thus we can assume that $|A-A|\le |A|^{1+\epsilon}$. Let $a$ be an arbitrary element in $A.$ Then we have 
\[|A-A|=|(A-a)-(A-a)|\le |A|^{1+\epsilon}.\]
Lemma \ref{m-dif2} gives us that for $|A|\le p^{9/16},$ 
\[|(A-a)(A-a)|\gtrsim |A|^{\frac{14}{9}-2\epsilon}.\]
Thus, if we apply Lemma \ref{m-dif1} with $X=(A-a)(A-a)$ and $B=(A-A)$, then we are able to obtain the following
\[|(A-a)(A-a)-(A-A)(A-A)|\gtrsim |A|^{1+\frac{7}{9}-\epsilon}\gtrsim |A|^{\frac{7}{4}+\frac{3}{180}},\]
where we used the condition that $|A|\le p^{9/16}.$ 

The same argument also works for the case of $G_2(A)$. Thus we leave the remaining details to the reader. This concludes the proof of the theorem.
\end{proof}
\section*{Acknowledgments}
D. Koh was supported by Basic Science Research Program through the National
Research Foundation of Korea(NRF) funded by the Ministry of Education, Science
and Technology(NRF-2015R1A1A1A05001374). T. Pham was supported by Swiss National Science Foundation grant P2ELP2175050. C-Y Shen was supported in part by MOST, through grant 104-2628-M-002-015 -MY4. The authors would like to thank Frank De Zeeuw for useful discussions. 

\end{document}